\documentclass{amsart}
\usepackage{amsmath, amssymb, mathdots}

\newtheorem{Lem}{Lemma}
\newtheorem{Theo}[Lem]{Theorem}
\newtheorem{Prop}[Lem]{Proposition}
\newtheorem{Cor}[Lem]{Corollary}
\newtheorem{Prob}{Problem}

\newcommand{\im}{\mathrm{im}}

\newcommand{\supp}{\mathrm{supp}\;}

\def\F{{\mathbb F}}

\def\R{{\mathbb R}}

\def\N{{\mathbb N}}

\begin{document}
\title[Branch groups, orbit growth, and subgroup growth]{Branch groups, orbit growth, and subgroup growth types for pro-$p$ groups}
\author[Y. Barnea]{Yiftach Barnea}
\address{
Department of Mathematics, Royal Holloway, University of London, Egham, Surrey, TW20 0EX, United Kingdom}
\email{y.barnea@rhul.ac.uk}
\author[J.-C. Schlage-Puchta]{Jan-Christoph Schlage-Puchta}
\address{
Institut f\"ur Mathematik, Ulmenstr. 69, 18051 Rostock, Germany}
\email{jan-christoph.schlage-puchta@uni-rostock.de}

\begin{abstract}
We first show that a class of pro-$p$ branch groups including the Grigorchuk group and the Gupta-Sidki groups all have subgroup growth type $n^{\log n}$. We then introduce the notion of orbit growth and use it to construct extensions of the Grigorchuk group and the Gupta-Sidki groups. We compute the subgroup growth type of these extensions and deduce that all functions between $n^{(\log n)^2}$ and $e^n$ occur as the subgroup growth type of a pro-$p$ group, thus, giving almost complete answer to a question raised by Lubotzky and Segal.
\end{abstract}

\maketitle

\section{Introduction and Results}

Branch groups form an important class of just infinite pro-$p$ groups, which includes the famous Grigorchuk group and the Gupta-Sidki groups. These groups have served as examples displaying quite unexpected behaviour. Although they have been extensively studied, the subgroup growth of these groups has not been determined up to now. The first goal of the present paper is to compute this invariant. 

For the rest of the this paper $p$ is a fixed prime and $\log n=\log_p n$. For a group $G$, we denote by $s_n(G)$ the number of subgroups of $G$ of index at most $n$ and denote by $d(G)$ the minimal number of generators of $G$, where if $G$ is a topological group we mean closed subgroups and topological generators respectively. Write $\Phi_p(G)=[G,G]G^p$, where $G^p$ is the (closed) subgroup generated by the $p$-powers in $G$. We denote by $d_{p}(G)=d \left( G/\Phi_p(G) \right)=\dim_{\F_p}\left( G/\Phi_p(G)\right)$. Note that if $G$ is a pro-$p$ group, then $\Phi_p(G)$ is its Frattini subgroup and $d(G)=d_p(G)$. Note also that if $\widehat{G}$ is the pro-$p$ completion of $G$, then $d(G) \geq d(\widehat{G})=d_p(\widehat{G})=d_p(G)$. If $G$ is the Grigorchuk group or a Gupta-Sidki group, then Pervova in \cite{Pervova} showed that $G$ has no maximal subgroups of infinite index, thus $\Phi_p(G)=\Phi(G)$, the Frattini subgroup of $G$, and $d_{p}(G)=d(G)$. This was generalized by Alexoudas, Klopsch and Thillaisundaram in \cite{multi spinal} to finite index subgroups of multi-edge spinal torsion group, in particular, the Grigorchuk group and the Gupta-Sidki groups.

For a pro-$p$ group $G$, the behaviour of $s_{p^n}(G)$ is intimately linked to the number of generators of subgroups. We denote by $d_G(m)$ the maximum of $d(U)$, as $U$ ranges over all subgroups of index
$p^m$. Similarly, for a group $G$ we denote by $d_{p,G}(m)$ the maximum of $d_p(U)$, as $U$ ranges over all subgroups of index $p^m$. In case it is clear from the context to which $G$ we refer, we will omit the subscript $G$. 

We say that $G$ has {\em logarithmic $p$-rank gradient}, if
$\alpha(G)=\lim\limits_{m\rightarrow\infty}\frac{d_p(m)}{m}$ exists, finite and positive.
Combining \cite[Proposition~1.6.2]{Subgroup Growth} and \cite[Lemma~4.2.1]{Subgroup Growth} it follows that
\begin{Prop}
\label{Prop:alpha properties}
Let $G$ be a $p$-group or a pro-$p$ group.
\begin{enumerate}
\item If $G$ has logarithmic $p$-rank gradient $\alpha$, then
\[
p^{\frac{\alpha^2}{4(\alpha+1)}m^2+o(m^2)} \leq
  s_{p^m}(G)\leq p^{\frac{\alpha}{2}m^2+o(m^2)}.
\]
\item If $H$ is a $p$-group or a pro-$p$ group respectively commensurable with $G$, then $H$ also has logarithmic $p$-rank gradient $\alpha$.
\item If $\mu\leq d_p(m-\mu)$, then
\[
p^{\mu(d_p(m-\mu)-\mu)}\leq s_{p^m}(G) \leq p^{\sum_{\nu=1}^{m-1} d(\nu)}
\]
\end{enumerate}
\end{Prop}

A group $G$ is called {\em self-replicating}, if there exists an integer $k$ and a normal
subgroup $G_1 \leq G$ of finite index, such that $G_1 \cong G^{(k)}$ the direct sum of $k$ copies of $G$. By iterating this process we obtain $G_n \cong G_{n-1}^{(k)} \cong G^{(k^n)}$ the $n$-th {\em principal congruence subgroup} of $G$. We say that $U<G$ is a {\em congruence subgroup}, if $U$ contains a principal congruence subgroup. If $G$ contains a self-replicating normal subgroup of finite index and every finite index subgroup of it is a congruence subgroup, then we say that $G$ has the {\em congruence subgroup property}. (We comment that this definition is slightly different than the standard definition for the congruence subgroup property for groups acting on trees.) With this notation our first result is the following.

\begin{Theo}
\label{thm:main}
Let $G$ be a finitely generated, self-replicating $p$-group or pro-$p$ group that has the congruence subgroup property. Then $G$ has subgroup growth of type $n^{\log n}$. More precisely, suppose $G$ contains a principal congruence subgroup $N$, such that $N \cong G^{(k)}$, $|G/N|=p^\ell$, and $N \leq \Phi_p(G)$. Put $d_p=\max d_p(H)$, where $H$ runs over all subgroups of the finite group $G/N$. Then we have that
\begin{enumerate}
\item $$\frac{(k-1)d_p(G)}{\ell}\leq\alpha(G)\leq (k-1)d_p(G)+d_p;$$
\item $$\alpha(G)=\lim\limits_{m\rightarrow\infty}\frac{d_p(m)}{m}=\sup\limits_{m\geq 0}
\frac{d_p(m)}{m+\frac{\ell}{k-1}}.$$
\end{enumerate}
\end{Theo}

From this we immediately obtain the following.
\begin{Cor}
\label{Cor:Grigorchuk growth}
If $G$ is the Grigorchuk group, then $$2^{\frac{9}{40}m^2+o(m^2)}\leq s_{2^m}(G)\leq 2^{6m^2+o(m^2)}.$$
If $G$ is the Gupta-Sidki $p$-group with $p\geq 3$, then $$p^{\frac{1}{8}m^2+o(m^2)}\leq s_{p^m}(G)\leq p^{\frac{3p^2-4p+1}{2}m^2+o(m^2)}$$.
\end{Cor}

The numerical bounds for $\alpha(G)$ are quite bad. In theory arbitrary good approximations to $\alpha(G)$ are computable, however, the na\"ive approach fails due to a horrendous amount of memory needed. We therefore ask the following.
\begin{Prob}
Give a numerical approximation to $\alpha(G)$, where $G$ is the Grigorchuk or a Gupta-Sidki group.
\end{Prob}

We comment that previously known examples of pro-$p$ groups with subgroup growth type $n^{\log n}$ (see below for the exact definition of subgroup growth type) are of different nature, that is, either linear and analytic groups or the Nottingham group and its index subgroups. Furthermore, many of these examples are just infinite, that is, their only non-trivial quotients are finite. Ershov and Jaikin-Zapirain in \cite{EJZ} constructed new hereditarily just infinite pro-$p$ groups. Moreover, these groups have subgroup growth type at least $n^{\left(\log n \right)^{2-\epsilon}}$, where $\epsilon$ is any positive number (private communication). However, no additional information is known about the subgroup growth type of these examples. 
We therefore suggest the following problem:
\begin{Prob}
What are the possible subgroup growth types of just infinite pro-$p$ groups? In particular,
compute the subgroup growth types of the examples constructed by Ershov and Jaikin-Zapirain. Furthermore, can a just infinite pro-$p$ group have exponential subgroup growth type?
\end{Prob}

Our next topic concerns the subgroup growth spectrum of pro-$p$ groups. Given a function $f:\N \rightarrow \R$ we say that $G$ has {\em subgroup growth of type $f(n)$} if there exist $a,b >0$ such that $s_n(G) \leq f(n)^a$ for all $n$ and $f(n)^b \leq s_n(G)$ for infinitely many $n$, if $f(n)^b \leq s_n(G)$ for all large $n$, then we say that $G$ has {\em strict subgroup growth of type $f(n)$}. If $G$ is a pro-$p$ group, we restrict $n$ to be a power of $p$ in the last condition. Segal in \cite{Segal} and Pyber in \cite{Pyber} showed that there are finitely generated profinite groups of any 'reasonable' subgroup growth type. On the other hand, Shalev showed in \cite{Shalev} that for a pro-$p$ group $G$ if there exists a constant $c < \frac{1}{8}$ such that $s_{n}(G) \leq n^{c\log n}$ for all $n$, then $G$ has {\em polynomial subgroup growth}, that is, its subgroup growth is of type $n$. Therefore, for pro-$p$ groups there is a gap in the subgroup growth between type $n$ and type $n^{\log n}$.

Pro-$p$ groups of types $n$, $n^{\log n}$ and $2^n$ are found in many natural examples. Segal and Shalev \cite{SSh} constructed metabelian pro-$p$ groups  with types $2^{n^{1/d}}$ for any $d \in \N$ and Klopsch in \cite{Klopsch} constructed metabelian pro-$p$ groups of types $2^{n^{(d-1)/d}}$ for any $d \in \N$, see \cite[Chapter 9]{Subgroup Growth}. No other types were discovered, hence, Lubotzky and Segal posed in \cite[Open Problems]{Subgroup Growth} the following problems:

\begin{Prob}
\begin{enumerate}
\item[]
\item[(a)] Are there any other gaps in the subgroup growth types of pro-$p$ groups?
\item[(b)] What other subgroup growth types occur for pro-$p$ groups?
\item[(c)] Is there an uncountable number of subgroup growth types (up to the necessary equivalence) for pro-$p$ groups?
\end{enumerate}
\end{Prob}

Here we show that almost all functions occur as the subgroup growth type of a pro-$p$ group. We give two different constructions. The first gives only a countable sequence of types, but the construction is effective and the subgroup growth type is strict. 

\begin{Theo}
\label{thm:growth type}
For every integer $k \geq 1$ and prime number $p$ there exists a residually finite $p$-group $H$ such that $s_{n}(H)$ is of strict type $n^{(\log n)^k}$. By taking the pro-$p$ completion of $H$ we obtain a pro-$p$ group with strict subgroup growth type $n^{(\log n)^k}$
\end{Theo}

The second construction is less effective, and gives inferior bounds, but applies to a larger range of functions.
\begin{Theo}
\label{thm:subgroup growth types}
Let $f:\N\rightarrow\N$ be a function, such that $f(n)\geq n^3$ and $\frac{f(n)}{p^n}\rightarrow 0$. Then there exists a residually finite $p$-group $H$ such that $s_{n}(H)$ is of type $e^{f(\log n)}$. By taking the pro-$p$ completion of $H$ we obtain a pro-$p$ group with subgroup growth type $e^{f(\log n)}$.
\end{Theo}
Notice that this means that we can obtain any subgroup growth type between $n^{(\log n)^2}=e^{(\log n)^3}$ and $e^n=e^{p^{\log n}}$ excluding $e^n$. However, we of course obtain subgroup growth type $e^n$ by considering a non-abelian free  pro-$p$ group. Hence, the only questions left considering the type is the following.
\begin{Prob}
What type of subgroup growth of pro-$p$ groups exist between $n^{\log n}$ and $n^{(\log n)^2}$? In particular, is there a gap in a growth types?
\end{Prob}
If one can improve the bound in Lemma~\ref{Lem:Module generation} in the special case of the Grigorchuk group and the Gupta-Sidki groups, then we can reduce the gap.

Comparing Theorem~\ref{thm:growth type} with Theorem~\ref{thm:subgroup growth types} we pose the following problem.

\begin{Prob}
Let $f$ be a function satisfying the assumptions of Theorem~\ref{thm:subgroup growth types}. Give an analytic condition which implies that there exists a pro-$p$ group $G$ of strict growth type $f$.
\end{Prob}
For any pro-$p$ group we have $d(n-1)-1\leq d(n)\leq p(d(n-1)-1)+1$, thus a function $f$ cannot vary too wildly, however, we expect that for all "reasonable" functions like $e^{n^\alpha}n^\beta(\log n)^\gamma$ with $0\leq\alpha<1$, $\beta, \gamma\in[0, \infty]$ the answer should be positive. We remark that the constructions by Segal and Shalev as well as the constructions by Klopsch mentioned above not only determine the subgroup growth type, but even the strict type.

Our construction depends on the fact that the Grigorchuk group and the Gupta-Sidki groups act on an element in the boundary of the $p$-regular rooted tree, that is, an infinite path, as transitive as pro-$p$ groups can, that is, if $U$ is a subgroup of finite index, then the number of $U$-orbits into which the tree decomposes grows only logarithmically with the index of $U$. To measure the transitivity of the action of a group define the {\em orbit growth} $o_n(G, X)$ of a group $G$ acting transitively on a set $X$ as the maximal number of orbits of a subgroup $U$ of index at most $n$ in $G$. Orbit growth and subgroup growth are linked by the following

\begin{Prop}
\label{Prop:Orbit vs subgroups}
Let $G$ be a $p$-group, acting transitively on a set $X$. Let $H$ be the wreath product $G\wr\F_p$ induced by this action. Then we have
\[
o_{p^n}(G, X)\leq d_{p,H}(n) \leq d_H(n) \leq o_{p^n}(G, X) + n\max_{m\leq n} d_{G}(m).
\]
\end{Prop}
Hence, the problem of constructing pro-$p$ groups of given subgroup growth type can be reduced to the problem of finding residually finite $p$-groups acting with given orbit growth type (as long as the orbit growth is faster than $nd_G(m)$). It is quite possible that in general Proposition~\ref{Prop:Orbit vs subgroups} is optimal, however, we believe that in the special case of branch groups with the congruence subgroup property the dependence on $d_G(m)$ can be improved. This would reduce the gap between $n^{\log n}$ and $n^{(\log n)^2}$ in Theorem~\ref{thm:subgroup growth types}.

\begin{Theo}
\label{thm:orbit growth}
Let $G$ be the Grigorchuk group or a Gupta-Sidki group. Let $f:\N\rightarrow\N$ be a non-decreasing function, and assume that $\frac{f(n)}{n}\rightarrow\infty$, $\frac{f(n)}{p^n}\rightarrow 0$. Then there exists a transitive action of $G$ on a set $X$, such that $o_{p^n}(G, X)\leq f(n)$ for all sufficiently large $n$, and $o_{p^n}(G, X)\geq\frac{1}{p}f(n)$ for infinitely many $n$.
\end{Theo} 

Theorem~\ref{thm:subgroup growth types} easily follows from Theorem~\ref{thm:orbit growth}, Proposition~\ref{Prop:Orbit vs subgroups}, and Theorem~\ref{thm:main}. In fact, from Theorem~\ref{thm:main} we have $d_{p, G}(n)=\mathcal{O}(n)$, thus from Proposition ~\ref{Prop:Orbit vs subgroups} we obtain $d_{p,H}(n)=o_{p^n}(G, X) + \mathcal{O}(n^2)$. By Theorem~\ref{thm:orbit growth} we see that we can choose $o_{p^n}(G, X)$ as we want, and Proposition~\ref{Prop:alpha properties} (3) then implies Theorem~\ref{thm:subgroup growth types}.

One of the main ingredients of our proof is the following.

\begin{Theo}
\label{thm:Grigorchuk orbits}
Let $G$ be the Grigorchuk group or a Gupta-Sidki group acting on the $p$-regular rooted tree $T$. Let $X$ be the orbit under $G$ of some infinite path in $T$. Then $o_{p^m}(G, X) \leq (p^5-1)m+1 \leq p^5m$. 
\end{Theo}

Here it is crucial that the number of orbits of a subgroup increases only slowly with the index. In fact, the logarithmic growth is the smallest possible and another remarkable property of branch groups.

\begin{Theo}
\label{thm:no smaller orbit growth}
Let $G$ be a $p$-group acting transitively on a set $X$. Suppose $o_n(G,X)$ is unbounded. 
Then there exist infinitely many $m$, such that  $$o_{p^m}(G,X) \geq (p-1)m+1.$$
\end{Theo}
We believe that orbit growth is of independent interest even for discrete groups. In analogy to Theorem~\ref{thm:no smaller orbit growth} we ask the following.
\begin{Prob}
Is there  a discrete group action with orbit growth less than logarithmic, in particular, is there one with orbit growth $\sim\frac{\log n}{\log\log n}$? Is there a slowest possible orbit growth for discrete groups?
\end{Prob}

Moreover, we wonder whether Theorem~\ref{thm:Grigorchuk orbits} is a special property of branch groups.
\begin{Prob}
Is there a class $\mathcal{V}$ of pro-$p$ groups, such that for all $G\in\mathcal{V}$ and all transitive actions of $G$ on a set $X$ we have that $o_{p^n}(G, X)$ is substantially larger than $n$? 
\end{Prob}

We notice that the maximal orbit growth possible is linear in the index. Indeed for any residually finite infinite group if we take a descending chain of normal subgroup and act on the coset tree we obtain linear growth. It is an interesting problem to find the possible orbit growth of a given group. For instance,
\begin{Prob}
Let $G$ be a $p$-adic analytic pro-$p$ group. What are the possible orbit growth of actions of $G$?
\end{Prob}
Notice that if $G$ is $p$-adic analytic and $n < o_{p^n}(G) < n^2$, then we get new types of subgroup growth. However, we do not know whether this is possible.

\vspace{.5cm}
\noindent \textbf{Acknowledgement.} We thank Alejandra Garrido for reading carefully a previous version of this paper and making an abundance of useful comments. We thank Alejandra Garrido and Benjamin Klopsch for organizing and inviting us to the conference "Trees, dynamics, and locally compact groups", where a significant part of this research took place.

\section{Proof of Theorem~\ref{thm:main}}
Since for a group $G$ we have that $d_p(G)=d(\widehat{G})$ and $s_{p^n}(G)=s_{p^n}(\widehat{G})$, where $\widehat{G}$ is the pro-$p$ completion of $G$ it suffices to assume $G$ is a pro-$p$ group in the proof of the theorem. We also recall that in such case $\Phi_p(G)=\Phi(G)$ and $d_p(G)=d(G)$.

The following is almost trivial, nevertheless, we include the proof for completeness since direct products often show unexpected behaviour.
\begin{Lem}
\label{Lem:direct}
Let $G$ and $H$ be pro-$p$-groups, $m$ a natural number, and assume that there is a constant $C$, such that for all subgroups $U_1<G$, $U_2<H$ with $(G:U_1)= p^k$, $(H:U_2)=p^\ell$ with $k,\ell\leq m$ we have $d(U_1)\leq Ck+d(G)$, $d(U_2)\leq C\ell+d(H)$. Then for all subgroups $U<G\times H$ with index $p^m$ we have $d(U)\leq Cm+d(G)+d(H)$.
\end{Lem}
\begin{proof}
Let $U$ be a subgroup of $G\times H$. Let $\pi:U\rightarrow H$ be the canonical projection. Then consider $U_1=U\cap G$, $U_2=\im\;\pi$. Take generators $g_1, \ldots, g_r$ of $U_1$, and generators $h_1, \ldots, h_s$ of $U_2$. For each $h_i$ choose a pre-image $\widetilde{h}_i$ under $\pi$. Then $g_1, \ldots, g_r, \widetilde{h}_1, \ldots, \widetilde{h}_s$ are contained in $U$, and generate a subgroup $\tilde{U}$ of $U$. We have $\tilde{U}\cap G=U\cap G$, and $\pi(\tilde{U})=\pi(U)$, thus, $(G\times H:\tilde{U})=(G\times H:U)$, therefore, $U=\tilde{U}$, and $d(U)\leq r+s$.

Since $(G\times H:U)=(G:U_1)(H:U_2)$ by our assumption we get that
\[
r+s\leq \log (G:U_1)+d(G)+\log(H:U_2)+d(H)=\log(G\times H:U)+d(G)+d(H),
\]
and our claim follows.
\end{proof}
\begin{Lem}
\label{Lem:inductive}
Let $G$ be a finitely generated self-replicating pro-$p$-group that has the congruence subgroup property. Define $N$, $k$
and $d$ as in Theorem~\ref{thm:main}. Then a subgroup of index $p^m$ has at
most $Cm+d(G)$ generators, where $C=(k-1)d(G)+d$.
\end{Lem}
\begin{proof}
We prove our assertion by induction on $m$. If $m=0$, then $U=G$, and our claim is trivial. Now let $U$ be an open
subgroup of index $p^m$. If $UN/N=G/N$, then in particular
$U\Phi(G)/\Phi(G)=G/\Phi(G)$, and we conclude that $U=G$ again. Henceforth we assume
that $UN/N<G/N$. Put $V=U\cap N$. Then $$(G:U)=(G:UN)(UN:U)=(G/N:UN/N)(N:V),$$
and thus, $(N:V)<(G:U)$.  As $V \leq N \cong G^{(k)}$ we can use Lemma~\ref{Lem:direct}
and the induction hypothesis to obtain that $d(V)\leq C(m-1)+k
d(G)$. Thus, $$d(U)\leq d(V)+d(UN/N)\leq d(V)+d.$$ We deduce that
\begin{multline*}
d(U)\leq C(m-1)+k d(G)+d = \\
((k-1)d(G)+d)(m-1)+((k-1)d(G)+d)+d(G)=
Cm+d(G),
\end{multline*}
and our claim is proven.
\end{proof}

\begin{proof}[Proof of Theorem~\ref{thm:main}.] We start with part (2). The inequality $\lim\frac{d(m)}{m}\leq\sup\frac{d(m)}{m+c}$ holds for any sequence $d(m)$ and for any real number $c$. For the reverse inequality we start by showing that for any $r \geq0$ we can find in $G$ a subgroup of index $p^{\ell r}$ which is isomorphic to $G^{((k-1)r+1)}$. This is done by induction on $r$. For $r=0$ we take $G$ itself. Suppose we know it for $r$, we will prove it for $r+1$. Take one of the components in the subgroup isomorphic to $G^{((k-1)r+1)}$ and replace it by $N$. The index of the new subgroup increases by $p^{\ell}$ so it is $p^{\ell (r+1)}$ while the number of components increases by $k-1$, so we obtain that the new subgroup is isomorphic to $G^{((k-1)(r+1)+1)}$ as required, thus, the induction is proved.

Let $U<G$ be a subgroup of index $p^m$ with $d(U)=d(m)$. By taking a subgroup isomorphic to $U$ in each component of $G^{((k-1)r+1)}$ we have that $G$ contains a subgroup $V \cong U^{((k-1)r+1)}$ of index $p^{\ell r}(G:U)^{(k-1)r+1}=p^{\ell r}(p^{m})^{(k-1)r+1}=p^{(\ell+m(k-1))r+m}$ with $d(V)=((k-1)r+1)d(U)=((k-1)r+1)d(m)$.

For a natural number $n \geq m$ take $r$ such that $(\ell+m(k-1))r+m \leq n < (\ell+m(k-1))(r+1)+m$, that is,
$r=\left\lfloor\frac{n-m}{\ell+m(k-1)}\right\rfloor$, pick a subgroup $V$ as described in the last paragraph, and choose $H<V$ with $(V:H)= p^{n-(\ell-m(k-1))r-m}$. Then $(G:H)=p^n$, and \\ $(V:H)<p^{\ell+m(k-1)}$, thus, $d(H) \geq d(V)-\ell-m(k-1)=((k-1)r+1)d(m)-\ell-m(k-1)$. Note that $m$, $k$, and $\ell$ are constant while $n$ and $r$ tend to infinity. We conclude that
\begin{multline*}
d(n)\geq d(H) \geq ((k-1)r+1)d(m)-\overbrace{\ell-m(k-1)}^{\mathcal{O}(1)} = \frac{\overbrace{(\ell+m(k-1))r}^{n+\mathcal{O}(1)}-\ell r}{m}d(m)+\mathcal{O}(1) \\ =  \frac{n-\ell r}{m}d(m)+\mathcal{O}(1).
\end{multline*}
Recall that $n=(\ell+m(k-1))r+\mathcal{O}(1)$. Thus, $\frac{\ell r}{n}=\frac{\ell}{\ell+m(k-1)}+\mathcal{O}\left({\frac{1}{n}}\right)$. Therefore, dividing the equations above by $n$ we obtain
\begin{multline*}
\frac{d(n)}{n} \geq \frac{n-\ell r}{nm}d(m)+\mathcal{O}\left(\frac{1}{n}\right)=
\frac{d(m)}{m}\left(1-\frac{\ell}{\ell+m(k-1)}\right)+\mathcal{O}\left(\frac{1}{n}\right)= \\
\frac{d(m)}{m}\frac{m(k-1)}{\ell+m(k-1)}+\mathcal{O}\left(\frac{1}{n}\right)=
\frac{d(m)}{\frac{\ell}{k-1}+m}+\mathcal{O}\left(\frac{1}{n}\right),
\end{multline*}
and our claim follows.

Finally note that the lower bound in Theorem~\ref{thm:main} (1) follows from Theorem~\ref{thm:main} (2) since
\[
\sup\limits_{m\geq 0} \frac{d(m)}{m+(\ell/(k-1))} \geq \frac{d(0)}{\frac{\ell}{k-1}} = \frac{(k-1)d(G)}{\ell}.
\]
\end{proof}

\begin{proof}[Proof of Corollary~\ref{Cor:Grigorchuk growth}.]
Suppose $G$ is the Grigorchuk group. Since $G$ is a $2$-group its profinite completion is the same as its pro-$2$ completion.
It follows from \cite[p.167, Proposition 8]{NewGri} that $G$ has a self-replicating subgroup $K$ of index 16, which is 3-generated. Write $K_i$ for the $i$ principal congruence subgroup of $K$. Then $K_1$ is of index 4 with $K_1 \cong K\times K$, and $K/K_1\cong C_4$.
Also, from \cite[p.169, Proposition 10]{NewGri} $G$ has the congruence subgroup property. Thus, we can assume that $G$ is a pro-$2$ group with the congruence subgroup property. Then $\Phi_p(K)$ contains $N=K_2 \cong K_1 \times K_1 \cong K^{(4)}$ of index $4^3=2^6$ in $K$. Hence, in the theorem we have $\ell=6$, $k=4$ and $d=3$, thus,
$\frac{3}{2}\leq \alpha(K) \leq 12$. From Proposition~\ref{Prop:alpha properties}
we obtain that the same inequality holds for $\alpha(G)$, and our claim follows.

Suppose $G$ is the Gupta-Sidki group $p$-group. Since $G$ is a $p$-group its profinite completion is the same as its pro-$p$ completion. The lower bound for the Gupta-Sidki groups is just the lower bound for non-$p$-adic analytic pro-$p$ groups, established by Shalev~\cite{Shalev}. Note that the following properties of the Gupta-Sidki groups were established by Sidki~\cite{Sidki} for $p=3$. These can be generalized easily for $p>3$, see Garrido~\cite[Section 2]{Alejandra} for proofs of most of these generalizations. For the upper bound we use the fact that $K$, the commutator of subgroup of $G$, is self-replicating, and we have that $\Phi_p(K)\geq K_2$, $K_1\cong K^{(p)}$, $(K:K_2)=p^{p^2-1}$, and $d(K)=p-1$.  Moreover, $K$ has the congruence subgroup property, and $d=p(p-1)$, the maximum being attained by the subgroup $K_1/K_2$ of $K/K_2$.  Hence, by Theorem~\ref{thm:main} we have $\alpha(G)\leq (k-1)d(K)+d=3p^2-4p+1$, and our claim follows from Proposition~\ref{Prop:alpha properties}.
\end{proof}
\section{Orbit growth and subgroup growth}

For the proof of Proposition~\ref{Prop:Orbit vs subgroups} we need the following results.
\begin{Lem}
\label{Lem:Module generation}
Let $G$ be an $m$-generated $p$-group, $M$ a $d$-generated $\F_pG$-module. Let $N$ be a submodule of $M$, which as an $\F_p$-vector space has codimension 1. Then $N$ is at most $(d+m-1)$-generated $\F_pG$-module.
\end{Lem}
\begin{proof}
Let $g_1, \ldots, g_m$ be generators of $G$, $v_1, \ldots, v_d$ be generators of $M$ as an $\F_pG$-module. Then $\left\{v_i^g|i\leq d, g\in G \right\}$ generates $M$ as an $\F_p$-vector space. Let $\varphi:M\rightarrow\F_p$ be the module homomorphism given by the canonical map $M\rightarrow M/N$. Notice that $\ker \varphi=N$. Since the action of a $p$-group on a cyclic group of order $p-1$ is trivial, we have that $\F_p$ is a trivial $\F_pG$-module, that is, for all $m\in M$ we obtain $\varphi(m^g)=\varphi(m)^g=\varphi(m)$.

Suppose without loss of generality that $\varphi(v_1)\neq 0$. We claim that as a vector space $N$ is generated by
\[
D=\left\{\varphi(v_i) v_1^g-\varphi(v_1) v_i^h|1\leq i\leq d, g, h\in G \right\}.
\]
As $\varphi(\varphi(v_i) v_1^g-\varphi(v_1) v_i^h)=0$ we have that $D \subseteq N$. Let $m\in N$ be an arbitrary element. Write $m=\sum_{i, j}\lambda_{ij} v_i^{h_j}$, where the $h_j \in G$. We have that
\[
\sum_{i, j}\lambda_{ij} v_i^{h_j} + \sum_{i,j}\frac{\lambda_{ij}}{\varphi(v_1)}\underbrace{\left(\varphi(v_i)v_1-\varphi(v_1)v_i^{h_j}\right)}_{\in D} = \left(\sum_{i,j}\frac{\lambda_{ij}}{\varphi(v_1)}\varphi(v_i)\right)v_1.
\]
The left-hand side is in $N$, and since $v_1$ is not in $N$, the right-hand side can only be in $N$ if it vanishes. But then we have represented $m$ as a linear combination of elements of $D$, that is, $D$ generates $N$ as a vector space.

We next claim that as an $\F_pG$-module $N$ is generated by the set
\[
\left\{v_1-v_1^{g_1}, v_1-v_1^{g_2}, \ldots, v_1-v_1^{g_m}, \varphi(v_i) v_1-\varphi(v_1) v_2, \ldots, \varphi(v_i) v_1-\varphi(v_1) v_d\right\}.
\]
Let $V$ be the $\F_pG$-module generated by this set. It suffices to show that $D\subseteq V$. We show first that $v_1-v_1^g\in V$ for all $g\in G$. To do so write $g$ as a word in $\{g_1^{\pm 1}, \ldots, g_m^{\pm 1}\}$, say $g=g_{i_1}^{\epsilon_1}\dots g_{i_k}^{\epsilon_k}$. Then
\[
v_1-v_1^g = v_1-v_1^{g_{i_k}^{\epsilon_k}} + \left(v_1-v_1^{g_{i_{k-1}}^{\epsilon_{k-1}}}\right)^{g_{i_k}^{\epsilon_k}} + \dots + \left(v_1-v_1^{g_{i_1}^{\epsilon_1}}\right)^{g_{i_2}^{\epsilon_2}\dots g_{i_k}^{\epsilon_k}},
\]
and since $v_1-v_1^{g^{-1}} = -\left(v_1-v_1^g\right)^{g^{-1}}$, we find that $v_1-v_1^g\in V$ for all $g\in G$.

Next for $g,h\in G$ we have $\varphi(v_i) v_1^g-\varphi(v_1) v_i^h\in V$, since
\[
\varphi(v_i) v_1^g-\varphi(v_1) v_i^h = \varphi(v_i)\left(v_1-v_1^{hg^{-1}}\right)^g+\left(\varphi(v_i)v_1-\varphi(v_1)v_i\right)^h.
\]
We have found a generating system consisting of $d+m-1$ elements, and our claim follows.
\end{proof}

\begin{proof}[Proof of Proposition~\ref{Prop:Orbit vs subgroups}]
We first give the upper bound. 
Let $U$ be a subgroup of $H$ of index $p^n$, let $B$ be the base group of the wreath product, and denote by $\pi$ the canonical projection $\pi:H\rightarrow G$.
Let $O_1, \ldots, O_N$ be a complete list of the orbits of $\pi(U)$. For each orbit pick an element $x_i\in O_i$, and define $b_i\in B$ to have $x_i$ coordinate $1$ and all other coordinates $0$. Then $b_1, \ldots, b_N$ generate $B$ as a $\pi(U)$-module. We conclude that $\pi^{-1}(\pi(U)) \geq U$ is generated by
\[
N+d(\pi(U))\leq N+\max_{m\leq n} d_G(m)
\]
elements.

Since $B$ is abelian we observe that the action of $U$ on $B$ factorizes through $\pi(U)$. Therefore, we need to bound the number of generators of $U \cap B$ as $\pi(U)$-module. As $\pi(U)$ is a $p$-group and $U \cap B$ is of finite index in $B$ there exists a sequence of $\pi(U)$-submodules $B=M_0>M_1>\dots>M_\ell=U\cap B$ with $(M_j:M_{j+1})=p$. We can repeatedly apply Lemma~\ref{Lem:Module generation} to find that for each $1 \leq j \leq \ell$ we have that $M_j$ can be generated by $\leq N+j(d(\pi(U))-1) \leq N+j\max_{m\leq n}d_G(m)$ elements. Hence, $U$ can be generated by $\leq N+n\max_{m\leq n} d_G(m)$ elements.

For the lower bound define a map $\varphi_i:U\rightarrow\F_p$, $1\leq i\leq N$, as follows. Write an element of $U$ as $(g,f)$, where $g\in G$ and $f:X\rightarrow\F_p$ has finite support. Then we define $\varphi_i:U \rightarrow \F_p$ by $\varphi_i((g, f)) = \sum_{x\in O_i} f(x)$. This map is a homomorphism since
\begin{multline*}
\varphi_i((g,f)(\tilde{g},\tilde{f})) = \varphi_i((g\tilde{g}, f^{\tilde{g}}+\tilde{f})) = \sum_{x \in O_i} f^{\tilde{g}}(x) + \sum_{x \in O_i} \tilde{f}(x)\\
 = \sum_{x \in O_i} f(x^{\tilde{g}}) + \sum_{x \in O_i} \tilde{f}(x) = \sum_{x \in O_i} f(x) + \sum_{x \in O_i} \tilde{f}(x) = \varphi_i((g,f)) + \varphi_i((\tilde{g},\tilde{f})).
\end{multline*}
Let $V=\F_p^N$, we define $\varphi:U \rightarrow V$ by $\varphi((g,f))=(\varphi_i((g,f)))_{i=1}^N$. For $i$ in the range $1\leq i\leq N$ fix an element $x_i \in O_i$. Given $\left( a_i \right) \in \F_p^{N}$ we define $f: X\rightarrow\F_p$ by $f(x)=a_i$, if $x=x_i$ for some $i\leq N$, and $f(x)=0$ otherwise. Then $\varphi(1,f)=\left( a_i \right)$. Thus, $\varphi$ is a surjective homomorphism.

We conclude that if $U$ is a subgroup of $G$, then $\pi^{-1}(U)$ maps onto $\F_p^N$. By the definition of orbit growth, there exists a subgroup $U$ of index $p^n$, which acts with $o_n(G, X)$ orbits, and $(G:U)=(H:\pi^{-1}(U))$, thus $\pi^{-1}(U)$ has index $p^n$ and maps onto $\F_p^{o_n(G)}$. Since $d_p(U)$ equals the maximal dimension of an $\F_p$-vector space onto which $U$ surjects, we obtain $o_n(G)\leq d_{p, H}(n)$. Finally the inequality $d_{p,H}(n)\leq d_H(n)$ always holds, and the proof is complete.
\end{proof}

Next we prove Theorem~\ref{thm:no smaller orbit growth}.

\begin{proof}[Proof of Theorem~\ref{thm:no smaller orbit growth}]
Pick a natural number $n$. Let $U$ be a finite index subgroup which acts with at least $n$ orbits on $X$. Pick a normal subgroup $N \triangleleft G$ of finite index within $U$. Write $\mathcal{O}$ for the set of orbits of $N$ acting on $X$. We have $|\mathcal{O}|\geq n$, as $N$ has at least as many orbits as $U$. Clearly $G/N$ acts transitively on $\mathcal{O}$. Choose  $o \in \mathcal{O}$, and let $H$ be the pre-image in $G$ of the stabilizer $(G/N)_o$. We will show that $H$ has at least $(p-1)\log (G:H)+1$ orbits. 

Notice that orbits of  $(G/N)_o$ on $\mathcal{O}$  are in one-to-one correspondence with the orbits of
$H$ on $X$. For instance, if $\left\{ o_1,o_2,\ldots, o_r \right\}$ is an orbit of $(G/N)_o$ on $\mathcal{O}$, then the corresponding orbit of $H$ on $X$ is $\cup_i  o_i$. In particular, the number of orbits of $H$ on $X$ equals the number of orbits of $(G/N)_o$ on $\mathcal{O}$. Let $\mathcal{O}_1, \ldots, \mathcal{O}_k$ be the orbits of $(G/N)_o$, and put $(G:H)=(G/N:(G/N)_{o})=p^m$. We have $p^m=|\mathcal{O}|\geq n$, in particular, by choosing $n$ appropriately, we can make $m$ arbitrarily large. Since $G/N$ is a $p$-group, we have for each $i$ that $|\mathcal{O}_i|=p^{m_i}$ for some $m_i\geq 0$. We conclude that $\sum_{i=1}^k p^{m_i}=p^m$. Note that at least one of the $m_i$ equals 0, since $(G/N)_o$ has the fixed point $o$, say $m_1=0$.

We claim that every integral solution of the equation $\sum_{i=1}^k p^{m_i}=p^m$ with $m_1=0$ satisfies $k \geq (p-1)m+1$. We prove our claim by induction over $m$. For $m=1$ the only integral solution is $1+1+\cdots+1=p$. Now assume our claim holds for $m-1$. Let $\mu_j$ be the number of indices $i$ with $m_i=j$. Then we have $\sum_j \mu_j p^j=p^m$. We have to show that $\mu_0\geq 1$ implies $\sum\mu_j\geq (p-1)m+1$. Let $(\mu_0, \ldots, \mu_m)$ be a tuple minimizing  $\sum\mu_j$. Then
$$
\mu_0 \equiv \sum_j \mu_j p^j=p^m \equiv 0 \pmod{p},
$$
together with $\mu_0\geq 1$ we conclude $\mu_0\geq p$. If $\mu_0\geq 2p$, then we could reduce $\mu_0$ by $p$ and increase $\mu_1$ by $\mu_1-1$, thus decreasing $\sum \mu_j$ by $p-1$, contradicting minimality. Hence, we may assume $\mu_0=p$. We now get $1+\sum_{j\geq 1} \mu_j p^{j-1} = p^{m-1}$, thus, by the inductive hypothesis we conclude $\sum_{j\geq 1} \mu_j \geq (p-1)(m-1)$. Together with $\mu_0=p$ we obtain $\sum_{j\geq 0}\mu_j\geq (p-1)m+1$, as claimed.
\end{proof}

\section{Proof of Theorem~\ref{thm:growth type} and \ref{thm:Grigorchuk orbits}}
\label{sec:growth type log^k}

For $p=2$ let $G$ be the first Grigorchuk group and let $K$ be the self-replicating subgroup of index 16 in it, while for $p>2$ let $G$ be the Gupta-Sidki $p$-group and let $K$ be the self-replicating subgroup of it of index $p^2$ (see \cite{Alejandra}). Let us emphasize that in this section we take $G$ and $K$ to be the discrete $p$-groups rather than their pro-$p$ completions. As mentioned in the introduction, by the work of Pervova \cite{Pervova} and Alexoudas, Klopsch and Thillaisundaram \cite{multi spinal} we do not have to distinguish between $\Phi$ and $\Phi_p$ and $d$ and $d_p$ for all finite index subgroups of $G$. Let $T$ be the $p$-regular rooted tree, $G$ acts naturally on $T$ and so does $K$. Let $T_n$ be the $n$-level of the tree and let $St(n)=St_G(T_n)$ be the stabilizer of $T_n$, where the root is considered as $T_0$, thus, $St(0)=G$. Note that $St(1)$ contains $K^{(p)}$ geometrically, that is, each component $K$ acts independently on the corresponding child tree of the first level. For $p=2$ this was proven by Grigorchuk~\cite[p. 167, Proposition~8]{NewGri}, for $p=3$ by Sidki~\cite{Sidki}, and for $p> 3$ by Fernandez-Alcober and Zugadi-Reizabal~\cite{AR}.

Clearly $K$ acts on the set of infinite paths in $T$. Pick one such path $x$, and let $X=\{x^g:g\in K\}$ be the orbit of $x$ under $K$.

\begin{proof}[Proof of Theorem~\ref{thm:Grigorchuk orbits}.]
We use induction on $m$ to prove the stronger claim that $U$ has at most $(p^5-1)m+1$ orbits on $X$.

If $m=0$, there is nothing to show, henceforth $m\geq 1$. We have that $\Phi_p(K)$ contains $St(5)$. For $p=2$ this was shown by Grigorchuk~\cite[p.168, Proposition~9] {NewGri}, for $p\geq 3$ Garrido~\cite[Proposition~2.6]{Alejandra} actually proved $\Phi_p(K)\geq St(4)$. Since $K$ is a $p$-group and $U$ is a proper subgroup of $K$ of finite index, we have $U\Phi_p(K) \neq K$ and hence $USt(5) \ne K$. Thus, we have $$(St(5):U \cap St(5))=(U St(5):U)<(K:U).$$ As $St(5)$ contains $K^{(p^5)}$ geometrically, we have that $$(K^{(p^5)}:U \cap K^{(p^5)})\leq (St(5):U \cap St(5))<(K:U).$$ Because $U \cap K^{(p^5)}$ is a subgroup of $U$ the number of orbits of $U$ on $X$ is bounded from above by the number of orbits of $U \cap K^{(p^5)}$.

For $1 \leq i \leq p^5$ let $X_i$ be the subset of $X$ consisting of the paths passing through the $i$-th vertex of the fifth level. Then $X$ is a disjoint union of the $X_i$ and each $X_i$ is invariant under $U \cap K^{(p^5)}$ because  $U \cap K^{(p^5)} \leq St(5)$. Thus, the number of orbits of $U \cap K^{(p^5)}$ acting on $X$ is the sum of the number of orbits of $U \cap K^{(p^5)}$ acting on $X_i$. Write $K^{(p^5)}=K_1\times\dots\times K_{p^5}$, where $K_i \cong K$. Let $\widetilde{U}_i$ be the projection of $U \cap K^{(p^5)}$ onto $K_i$. The action of $U \cap K^{(p^5)}$ on $X_i$ factors through $\widetilde{U}_i$ so the number of orbits of $U \cap K^{(p^5)}$ acting on $X_i$ is the number of orbits of $\widetilde{U}_i$ acting on each $X_i$.

For $1\leq i\leq p^5$ let $U_i$ be the projection of $U \cap (K_1\times \dots \times K_{i})$ onto $K_i$. Then $U_i$ is a subgroup of $\widetilde{U}_i$, so the number of orbits of $\widetilde{U}_i$ acting on $X_i$ is bounded above by the number of orbits of $U_i$ acting on $X_i$.

Using induction on $r$ it is easy to prove that $(K_1 \times\dots\times K_{r}:U \cap (K_1 \times\dots\times K_{r}))=\prod_{i=1}^{r} (K_i:U_i)$. We deduce that
$$\prod_{i=1}^{p^5} (K_i:U_i)=(K^{(p^5)}:U \cap K^{(p^5)})=(UK^{(p^5)}:U) \leq (U\Phi(K):U) < (K:U),$$ in particular, each single factor on the left is strictly smaller than $p^m$. Viewing $U_i$ as a subgroup of $K_i$ acting on $i$-th subtree we can apply our induction hypothesis to find that the number of orbits of $U\cap K^{(p^5)}$ acting on $X$ is at most
\begin{multline*}
\sum_{i=1}^{p^5} \left((p^5-1)\log(K_i:U_i)+1\right) = (p^5-1)\log(K^{(p^5)}:U \cap K^{(p^5)}) +p^5  \\ \leq (p^5-1)(m-1)+p^5 = (p^5-1)m+1,
\end{multline*}
and the proof is complete.
\end{proof}

By taking componentwise action we obtain for $k \geq 2$ a transitive permutation group $(K^{(k)}, X^{(k)})$. For a set $Y$ we write $\F_p^{\left(Y \right)}$ for the direct sum of $\F_p$ indexed by $Y$. Let $H_k$ be the restricted wreath product $\F_p \wr (K^{(k)}, X^{(k)}) \cong \F_p^{\left(X^{(k)} \right)} \rtimes K^{(k)}$. Then $H_k$ is  a finitely generated $p$-group. Let $B$ be the base group of the wreath product and $\pi:H_k \rightarrow K^{(k)} $ be the projection onto the active group.

In general, wreath products of residually finite groups need not be residually finite. However, in all situations we are interested in, residually finiteness can be obtained from the following. 
\begin{Lem}
\label{Lem:residually finite}
Let $\Gamma$ be a residually finite group, $\Omega$ a set on which $\Gamma$ acts transitively. Suppose that for all $x,y\in \Omega$, there exists a finite index subgroup $\Delta$ such that $x$ and $y$ are not in the same $\Delta$-orbit. Then the restricted wreath product $\F_p \wr \Gamma$ given by the action is residually finite.
\end{Lem}
\begin{proof}
Let $(f,\gamma)$ be a non-trivial element of the wreath product, where  $f:\Omega \rightarrow\F_p$ is a  function with a finite support and $\gamma \in \Gamma$. If $\gamma \neq 1$, then there exists a subgroup $\Delta<\Gamma$ of finite index not containing $\gamma$, and the pre-image of $\Delta$ under the canonical projection is a finite index subgroup not containing $(f,\gamma)$. Now suppose that $\gamma=1$, $f$ does not vanish identically, and has support $\supp f = \{x_1, \ldots, x_n\} \neq \emptyset$. For $1 \leq i < j \leq n$ choose a finite index subgroup $\Delta_{ij}$, such that $x_i$ and $x_j$ are not in the same $\Delta_{ij}$-orbit. Then $\Delta=\bigcap \Delta_{ij}$ is a finite index subgroup, which acts on $\Omega$ with finitely many orbits $\Omega_1, \ldots, \Omega_N$. Then
\[
M=\{m: \Omega \rightarrow \F_p: | \supp m|< \infty, \forall i \leq N: \sum_{x \in \Omega_i} m(x)=0\}
\]
 is $\Delta$-invariant and of finite index in $\F_p^{(\Omega)}$, hence $\{(m,\delta) \in \F_p \wr \Gamma:  m \in M, \delta \in \Delta \}$ is a finite index subgroup of $\Gamma$, which does not contain $(f,\gamma)$, because each $\Omega_i$ contains at most one point on which $f$ does not vanish, and $f$ does not vanish identically.
\end{proof}

\begin{Cor}
$H_k$ is residually finite.
\end{Cor}
\begin{proof}
Let $x, y\in X$ be different paths. In view of Lemma~\ref{Lem:residually finite} we have to show that there exists a finite index subgroup $U$ of $G$, such that $x$ and $y$ are in different orbits of $U$. But if $x$ and $y$ are different, there is some $\ell$ such that there intersection with the $\ell$-th level is different, and $St(\ell)$ does what we need.
\end{proof}

We now would like to show that $o_{p^n}(G^{(k)},X^{(k)})$ is of magnitude $n^{k}$. We start with a lower bound.
\begin{Lem}
\label{Lem:Bound for congruence index}
Let $\Gamma$ be a $p$-group acting on $T$. Then for every natural number $\ell$ we have that $(\Gamma:St_{\Gamma}(\ell)) \leq p^{p^\ell}$.
\end{Lem}
\begin{proof}
There is an injective homomorphism $\Gamma/St_{\Gamma}(\ell)\rightarrow S_{p^\ell}$. The image of this homomorphism is a $p$-subgroup, hence a subgroup of the $p$-Sylow subgroup of $S_{p^\ell}$. The latter has order $p^{\frac{p^\ell-1}{p-1}}$, and our claim follows.
\end{proof}

\begin{Lem}
\label{Lem:Many orbits}
Let $G$ be the Grigorchuk group or a Gupta-Sidki group. We have that $(G^{(k)}:St(\ell)^{(k)})\leq p^{kp^\ell}$, and $St(\ell)^{(k)}$ has at least $p^{k\ell}$ orbits on $X^{(k)}$. In particular,  $$o_{p^n}(G^{(k)}, X^{(k)}) \geq  \frac{1}{(pk)^k} n^k.$$
\end{Lem}
\begin{proof}
The bound for the index follows immediately from Lemma~\ref{Lem:Bound for congruence index}.
Now recall that $G^{(k)}$ acts transitively on $T_{\ell}^{(k)}$, whereas $St(\ell)^{(k)}$ acts trivially on $T_\ell^{(k)}$. Thus, if $(x_1, \ldots, x_k), (x_1', \ldots, x_k') \in X^{(k)}$ are in the same $St(\ell)^{(k)}$-orbit, we have for all $i$ that $x_i$ and $x_i'$ pass through the same point in level $\ell$. Therefore, the number of orbits is at least the number of $k$-tuples of vertices of level $\ell$ which is $(p^{\ell})^k=p^{k\ell}$. 
We conclude that
\[
o_{p^n}(G^{(k)}, X^{(k)})\geq \max_{\left\{ \ell :  kp^\ell\leq n \right\}} p^{k\ell} \geq p^{k(\log(n/k)-1)} = \frac{1}{(pk)^k}p^{k\log n}=\frac{1}{(pk)^k}n^k,
\]
proving the last claim of the lemma.
\end{proof}

One can give more precise bounds for the index of $St(\ell)$. Doing so is essentially equivalent to determining the Hausdorff dimension of $G$, which was done for $p=2$ by Grigorchuk~\cite[Section~5]{NewGri}, and for $p\geq 3$ in vast generality by Zugadi-Reizabal~\cite{Reizabal}.

Now we turn to upper bounds for $o_{p^n}(G^{(k)}, X^{(k)})$.

\begin{Lem}
\label{Lem:permutation product}
Let $(\Gamma, \Omega)$, $(\Delta, \Lambda)$ be two permutation groups. Let $U$ be a subgroup of $\Gamma \times \Delta$. Suppose that $U \cap \Gamma$ has $m$ orbits on $\Omega$, and $U\Gamma/\Gamma<\Delta$ has $n$ orbits on $\Lambda$. Then $U$ has at most $mn$ orbits on $\Omega \times \Lambda$.
\end{Lem}
\begin{proof}
Let $x_1, \ldots, x_m$ be representatives of the orbits of $U\cap \Gamma$ acting of $\Omega$, and $y_1, \ldots, y_n$ be representatives of $U\Gamma/\Gamma$ acting on $\Delta$. Then by considering the action of $U\Gamma/\Gamma$ on $\Delta$ every element $(x,y) \in \Omega \times \Lambda$ is equivalent to an element of the form $(z, y_j)$, $1 \leq j\leq n$. Take an element $u \in U\cap \Gamma$ such that $z^{u}=x_i$ for some $1 \leq i\leq m$. By applying it to $(z, y_j)$ we obtain that $(x,y)$ is equivalent $(x_i, y_j)$.
\end{proof}

\begin{Lem}
\label{Lem:Few orbits}
Let $U$ be a subgroup of $K^{(k)}$ of index $p^n$. Then $U$ has at most $\mathcal{O}(n^k)$ orbits on $X^{(k)}$.
\end{Lem}
\begin{proof}
We prove our claim by induction on $k$. The case $k=1$ is Theorem~\ref{thm:Grigorchuk orbits}. Now suppose that $k\geq 2$, and our claim is already shown all smaller values of $k$. Write $K^{(k)}=K_1\times\dots\times K_k$. Assume that $(UK_1:U)=(K_1:U\cap K_1)=p^m$. Then from the induction hypothesis for $k=1$ we deduce that $U\cap K_1$ acts with $\mathcal{O}(m)$ orbits on $X$. Notice that $UK_1/K_1$ acts on $X^{(k)}/K_1\cong X^{(k-1)}$, and that $(K^{(k)}/K_1:UK_1/K_1)=(K^{(k)}:UK_1)=\frac{(K^{(k)}:U)}{(UK_1:U)}=p^{n-m}$, applying the induction hypothesis for $k-1$ we find that the number of orbits of this action is $\mathcal{O}((n-m)^{k-1})$. Lemma~\ref{Lem:permutation product} implies that $U$ itself has $\mathcal{O}(m(n-m)^{k-1}) = \mathcal{O}(n^k)$ orbits, and the proof of the general case is also complete.
\end{proof}

\begin{proof}[Proof of Theorem~\ref{thm:growth type}.]
For $k=1$, this is Corollary~\ref{Cor:Grigorchuk growth}.
From Lemma~\ref{Lem:Many orbits} and \ref{Lem:Few orbits} we obtain $$c_1 n^k \leq o_{p^n}(G^{(k)}, X^{(k)})\leq c_2n^k$$ for some positive constants $c_1$ and $c_2$. Theorem~\ref{thm:main} implies $d_{G^{(k)}}(n)\leq Cn$. Plugging these values into Proposition~\ref{Prop:Orbit vs subgroups} we obtain $c_1 n^k\leq d_{p,H_k}(n)\leq c_2 n^k + Cn^2$, which for $k\geq 2$ implies our claim. 
\end{proof}

\section{Trees with a single infinite path}

In this section we consider subtrees $S\subseteq T$, which consist of a single infinite path and a bunch of finite branches spreading from this path. We assume that all vertices of $S$ are coloured with distinct colours. The action of $G$ on $T$ induces an action of $G$ on the set of all embeddings of $S$ into $T$. Here embeddings, which differ in the colouring, are considered different. In particular, if $g\in G$ induces a non-identity automorphism on $S$ as a rooted tree, then $g$ does not fix the embedding of $S$ as a coloured tree. Thus, for example, if we take the coloured subtree $R$ consisting of the root and all its children, then its orbit under $G$ consist of $p$ coloured subtrees, while if we do not colour it we have that it is invariant under $G$. 

The following theorem implies Theorem~\ref{thm:orbit growth}.
\begin{Theo}
\label{thm:all orbit types}
Let $f:\N \rightarrow \N$ be a function, such that $\frac{f(n)}{n} \rightarrow \infty$ and $\frac{f(n)}{p^n}\rightarrow 0$. Then there exists a tree $S \subseteq T$, which has a single infinite path, such that for all $n$ large enough and all $U<G$ with $(G:U)=p^n$ we have that $U$ acts with at most $f(n)$ orbits on $S^G$, and there exist infinitely many $n$ and subgroups $U$ with $(G:U)=p^n$, such that $U$ acts with at least $\frac{1}{p}f(n)$ orbits on $S^G$. 
\end{Theo}

The subgroups determining the lower bound of the orbit growth in the theorem will be certain stabilisers of anti-chains in $T$. Before proving the theorem we study such stabilisers in general.  If $\mathcal{A}=\{x_1, \ldots, x_k\}$ is a maximal anti-chain in $T$, then we denote by $S_\mathcal{A}$ the finite tree obtained by deleting all vertices below any element of $\mathcal{A}$. Denote by $G_\mathcal{A}$ the pointwise stabiliser of $\mathcal{A}$ in $T$. For example if $\mathcal{A}$ is the $\ell$-level, then 
$S_\mathcal{A}=T_{\ell}$ and $G_\mathcal{A}=St(\ell)$.

We say that a subtree $S$ is complete, if for all vertices $v\in S$ and all siblings $w$ of $v$ we have $w\in S$ as well. The completion $\overline{S}$ of $S$ is the smallest complete subtree of $T$ containing $S$. We extend the colouring of $S$ to $\overline{S}$ in such a way that the vertices in $\overline{S}\setminus\overline{S}$ get different colours not occurring in $S$.

\begin{Lem}
\label{Lem:completion properties}
Let $S$ be a tree, and denote by $\overline{S}$ the completion of $S$.
\begin{enumerate}
\item We have $\overline{S^g}=(\overline{S})^g$ for every $g\in G$.
\item Let $U$ be a subgroup of $G$. Then the map $S\mapsto\overline{S}$ induces a bijection between the $U$-orbits on $S^G$ and the $U$-orbits on $\overline{S}^G$.
\end{enumerate} 
\end{Lem}
\begin{proof} 
Both $\overline{S^g}$ and $(\overline{S})^g$ contain $S^g$. If $v\in(\overline{S})^g$, and $w$ is a sibling of $v$, then $v^{g^{-1}}\in S$, and $w^{g^{-1}}$ is a sibling of $v^{g^{-1}}$, thus, $ w^{g^{-1}}\in \overline{S}$, and therefore $w\in(\overline{S})^g$. Hence, $(\overline{S})^g$ is complete, and $\overline{S^g}\subseteq(\overline{S})^g$.

If $v\in (\overline{S})^g$, then there is some $w\in S$, such that $v^{g^{-1}}$ is a sibling of $w$, thus, $w^g\in S^g$ is a sibling of $v$, and therefore $v\in\overline{S^g}$. We obtain $(\overline{S})^g\subseteq\overline{S^g}$, so $\overline{S^g}=(\overline{S})^g$. 

To prove the second statement note that an orbit $\{S^{gu}:u\in U\}$ is mapped to $\{\overline{S^{gu}}:u\in U\}$. Using the first statement we have $\{\overline{S^{gu}}:u\in U\} = \{\overline{S^g}^u:u\in U\}$, that is, the image of an orbit is again an orbit. The map is clearly surjective, we now show that it is injective. Suppose that $\{\overline{S^{gu}}:u\in U\} = \{\overline{S^{hu}}:u\in U\}$. Then there exists some $u_0$, such that $\overline{S^{g}}=\overline{S^{hu_0}}$. Since we can reconstruct $S^g$ from $\overline{S^g}$ by deleting all vertices which have the same colour as a vertex in $\overline{S}\setminus S$,  and the action of $u_0$ preserves colours, we obtain $S^g=S^{hu_0}$, and therefore $\{S^{gu}:u\in U\}=\{S^{hu}:u\in U\}$. Hence, $S\mapsto \overline{S}$ induces a bijection of orbits.
\end{proof}

\begin{Lem}
\label{Lem:rigid stabilizer}
Let $G$ be the Grigorchuk or a Gupta-Sidki group. For a vertex $v$ the restriction of the stabiliser $G_v$ of $v$ to the tree $S_v$ consisting of all vertices below $v$ is isomorphic to $G$.
\end{Lem}
\begin{proof}
We prove the assertion by induction over the level $\ell$ of $v$. Assume first that $\ell=1$.

If $p=2$, then $G=\langle a,b,c,d\rangle$, where $a$ interchanges the points of level 1, and $b=(a,c)$, $c=(a,d)$, $d=(1,b)$. Here we adopt the notation from \cite{NewGri}, where for automorphisms $\varphi_1, \ldots, \varphi_p$ of $T$ we denote by $(\varphi_1, \ldots, \varphi_p)$ the automorphism fixing the uppermost layer of $T$, and acting by $\varphi_i$ on $S_{v_i}$, where $v_i$ is the $i$-th vertex in the uppermost layer. From this we see that $G_{\mathcal{L}}=\langle b,b^a,c,c^a,d,d^a \rangle \leq G^{(2)}$ is obvious, and the restriction is a subgroup of $G$. Assume without loss of generality that $v$ is the right point on the first level. Then the restriction of $b, c, d$ to $S_v$ are $c, d, b$, respectively. Furthermore, $a$ is the restriction of $b^a$, thus the restriction contains $\langle a,b,c,d\rangle=G$, and therefore equals $G$.

For $p\geq 3$ we have $G=\langle a, b\rangle$, where $a$ is the cyclic permutation of the $p$ vertices of level 1, and $b=(a, a^{-1}, 1, \ldots, 1, b)$.  Again $G_{\mathcal{L}}\leq G^{(p)}$ is clear, and the restriction to the rightmost vertex contains the restriction of $b$, which is $b$, and the restriction of $b^{a^{-1}}$, which is $a$. Thus the restriction equals $\langle a,b\rangle = G$. 

In general let $v$ be a vertex of level $\ell$, $w$ the parent of $v$. The restriction of $G_w$ to the subtree with root $w$ is isomorphic to $G$ by the inductive hypothesis, and the restriction of $G_v$ equals the restriction of $(G_w)_v$. Since $v$ is of level 1 in the subtree with root $w$, the restriction of $G_v$ to the tree with root $v$ is isomorphic to $G$ by the case $\ell=1$ already proven.
\end{proof}

\begin{Lem}
\label{Lem:continuity of orbit numbers}
Suppose that $S' \subset S\subseteq T$ are complete trees, such $S \setminus S'$ consists of $p$ vertices, which are siblings of each other. Let $X, X'$ be the orbits of $S, S'$ under $G$, respectively. Then for all subgroups $U$ of $G$ we have that the number of orbits of $U$ on $X$ is at most $p$ times the number of orbits of $U$ on $X'$, and at least equal to the number of orbits of $U$ on $X'$.
\end{Lem}
\begin{proof}
Put $S\setminus S'=\{x_1, \ldots, x_p\}$. 
Define a map $\psi:S_{\mathcal{L}}^G \rightarrow {S'_{\mathcal{L}}}^G$ by mapping $S_{\mathcal{L}}^g$ to $S_{\mathcal{L}}^g\setminus\{x_1^g, \ldots, x_p^g\}$. This map is well defined, that is, it does not depend on the choice of $g$, as the points $x_1, \ldots, x_p$ are detected by their colours, which are preserved by $g$. For the same reason $\psi((S_{\mathcal{L}}^g)^u) = \psi(S_{\mathcal{L}}^g)^u$, thus $\psi$ induces a map from the set of $U$-orbits on $S_{\mathcal{L}}^G$ to the set of $U$-orbits on ${S'_{\mathcal{L}}}^G$. The pre-image of an orbit $({S'_{\mathcal{L}}}^g)^U$ consists of $p$ orbits or $1$ orbit, depending on whether the stabilizer of ${S'_{\mathcal{L}}}^g$ in $U$ fixes $x_1, \ldots, x_p$ or not. In any case, the number of orbits of $U$ on $S_{\mathcal{L}}^G$ is at least equal to the  number of orbits of $U$ on ${S'_{\mathcal{L}}}^G$, and at most $p$ times this number.
\end{proof}

We say that $U<G$ is a subgroup of {\em level $\ell$} if $St(\ell) \leq U$.
\begin{Lem}
\label{Lem:antichain orbits}
Let $U<G$ be a subgroup of index $p^n$ and level $\ell$. Let $\mathcal{L}$ be the antichain consisting of all vertices of level $\ell$. Assume that $S$ is complete and that $S_\mathcal{L}$ contains $k$ vertices. Then $U$ acts with at most $np^{\frac{k-p \ell-1}{p}+5}$ orbits on $X$.
\end{Lem}
\begin{proof}
Because $U$ is of level $\ell$ a subset $Y$ of $X$ meets every orbit of $U$ if and only if the image of $Y$ under the restriction $S\mapsto S_\mathcal{L}$ meets the image of every orbit of $U$ under this map. Hence, the number of orbits does not depend on the shape of $S$ below level $\ell$, so we may assume without loss of generality that below level $\ell$ $S$ consists of an infinite path and the siblings of the vertices on this infinite path. 

We now prove the lemma by induction on $k$. We first choose where the infinite path of $S$ goes to. We notice that then the image of each siblings of an element on this path is determined. By Theorem~\ref{thm:Grigorchuk orbits} and Lemma~\ref{Lem:completion properties}~(2) this can be done in at most $p^5n$ ways. If $k=\ell p+1$, then $S$ contains no further vertices, and our claim is true. If $k<\ell p+1$, then $S$ does not contain an infinite path, contrary to our assumption.

Now suppose that $k>\ell p+1$ and our claim holds for all trees $S'$ such that $S'_\mathcal{L}$ has less than $k$ vertices. Let $x_1$ be a vertex in $S$ with the largest possible level that does not exceed $\ell$ such that its parent $y$ is not on the infinite path. Such a vertex exists since $k>\ell p+1$. Then $x_1$ and its siblings $x_2, \ldots, x_p$ are  leaves in $S$. Let $S'$ be the tree that is obtained from $S$ by removing $x_1, \ldots, x_p$ and their descendeds. 

By Lemma~\ref{Lem:continuity of orbit numbers} the number of orbits of $U$ on $S^G$ is at most $p$ times the number of orbits of $U$ on ${S'}^G$. By our inductive hypothesis the latter is at most
\[
np^{\frac{|S'_{\mathcal{L}}|-\ell p-1}{p}+5} = np^{\frac{|S_{\mathcal{L}}|-p-\ell p-1}{p}+5} = np^{\frac{k-\ell p-1}{p}+4},
\]
and our claim follows.
\end{proof}

\begin{proof}[Proof of Theorem~\ref{thm:all orbit types}.] Let $f:\N\rightarrow\N$ be a function, and suppose that $\frac{f(n)}{n}\rightarrow\infty$, $\frac{f(n)}{p^n}\rightarrow 0$. We will construct a tree $S$ together with an increasing sequence of integers $(n_i)$, such that the orbit growth of $G$ with respect to this action is strictly bounded by $f$, and for all $i$ there exists some $m_i \in[n_{i-1}, n_i]$ and a subgroup $U_{i}<G$ of index $p^{m_i}$, such that $U_i$ has at least $\frac{f(m_i)}{p}$ orbits on $X$.

We would like to construct recursively a sequence of complete coloured subtrees $S_i$ with $S_i \subset S_{i+1}$ for all $i$ such that $S_i \setminus S_1$ is finite and a strictly increasing sequence $(n_i)$ with the following properties:
\begin{enumerate}
\item[(i)] For all $n \geq n_1$, the action of $G$ on $S_i^G$ has $o_{p^n}(G,S_i^G) \leq f(n)$. 
\item[(ii)] For all $k\geq 2$ there exists some $m_k  \in [n_{k-1}, n_k]$ and a subgroup $U_k$ of index $p^{m_k}$, such that for all $i \geq k$ we have that $U_k$ acts on $S_i^G$ with at least $\frac{f(m_k)}{p}$ orbits.
\end{enumerate} 
We let $S_1$  be the completion of some infinite path. By Theorem~\ref{thm:Grigorchuk orbits} and the assumption that $\frac{f(n)}{n}\rightarrow\infty$ we can choose $n_1$ in such a way that for all $n \geq n_1$, $o_{p^n}(G,S_1^G) \leq f(n)$.

Suppose we constructed $S_j$ and $n_j$ for all $j<i$.
If we add (coloured) vertices to $S_{i-1}$, the from Lemma~\ref{Lem:continuity of orbit numbers} the orbit growth cannot decrease. Let $\ell$ be the maximum between the maximal level of a subgroup of index at most $p^{n_{i-1}}$ and the maximal level of a vertex in $S_{i-1}\setminus S_1$. If we only add vertices below level $\ell$ to $S_{i-1}$, then the orbit growth up to $n_i$ does not change. Let $N(m)=(G:St(m))$. We would like to pick a number $\ell'>\ell+2$, such that 
\begin{equation}
\label{eq: ell definition}
\mbox{for all $t\geq N(\ell'-\ell)$ we have\ }
\frac{t}{N(\ell)} > f\left(\log t\right).
\end{equation} To see that such an integer $\ell'$ exists, assume  that there exist arbitrary large $t$, such that $\frac{t}{N(\ell)} \leq f(\log t)$. Then for arbitrary large $n=p^t$ we have $f(n)\geq\frac{p^n}{N(\ell)}$, contrary to the assumption $\frac{f(n)}{p^n}\rightarrow 0$. Hence, there exists some $t_0$, such that $\frac{t}{N(\ell)} > f\left(\log t\right)$ holds for all $t>t_0$. Since $N(\ell'-\ell)$ tends to infinity with $\ell'$, we can pick any $\ell'>t_0$.

Since $f$ grows faster than linearly we can choose $n_i$ such that for all $n>n_i$
\begin{equation}
\label{eq:n_i definition}
np^{\frac{p^{\ell'-\ell}-1}{p(p-1)}+ \frac{p^{\ell+1}-1}{p(p-1)}+5} < f(n).
\end{equation}
The reasons for these choices will become clear in a moment, for now it is sufficise to picture these numbers as much larger than $\ell$ and $n_{i-1}$, respectively.

We will construct $S_i$ by adding vertices to $S_{i-1}$ between level $\ell+2$ and $\ell'$. More specifically, let $v_i$ be the unique vertex of level $\ell+1$ on the infinite path of $S_1$. Then define $\mathcal{S}$ as the set of all complete trees $R$ containing $S_{i-1}$ and such that all vertices of $R \setminus S_{i-1}$ are descendants of $v_i$ and have level at most $\ell'$. We will eventually choose $S_i$ within $\mathcal{S}$.

We first notice that from the choice of $\ell$ it follows that $o_{p^n}(G, S_{i-1}^G)=o_{p^n}(G, R^G)$ for all $n\leq n_{i-1}$ and for all $R\in\mathcal{S}$. On the other hand, we claim that $o_{p^n}(G, R^G)< f(n)$ for all $n>n_i$ and all $R\in\mathcal{S}$.  In fact, for $R\in\mathcal{S}$, all vertices in $R\setminus S_{i-1}$ are descendants of $v_i$, and therefore, $|R\setminus S_{i-1}| \leq \frac{p^{\ell'-\ell}-1}{p-1}$, and $|R\setminus S_1|\leq \frac{p^{\ell'-\ell}-1}{p-1} + \frac{p^{\ell+1}-1}{p-1}$. Let $U$ be a subgroup of index $p^n$ and level $\lambda$. Then we can apply Lemma~\ref{Lem:antichain orbits} to find that the number of orbits of $U$ on $R^G$ is at most $np^{\frac{k-p\lambda-1}{p}+5}$, where $k=|R_\mathcal{L}|$ and $\mathcal{L}$ is the antichain consisting of all vertices on level $\lambda$. We have 
\[
k=|R_\mathcal{L}|= |(S_1)_\mathcal{L}|+|(R\setminus S_1)_\mathcal{L}| \leq  |(S_1)_\mathcal{L}|+|R\setminus S_1| \leq  p\lambda+1+\frac{p^{\ell'-\ell}-1}{p-1} + \frac{p^{\ell+1}-1}{p-1},
\]
so
\[
k-p\lambda-1 \leq  \frac{p^{\ell'-\ell}-1}{p-1} + \frac{p^{\ell+1}-1}{p-1}.
\]
Therefore, the number of orbits  of $U$ on $R^G$ is at most
\[
np^{\frac{k-p\lambda-1}{p}+5} \leq np^{\frac{p^{\ell'-\ell}-1}{p(p-1)} + \frac{p^{\ell+1}-1}{p(p-1)}+5}.
\]
By the choice of $n_i$ in (\ref{eq:n_i definition}) we conclude that $o_{p^n}(G, R^G)< f(n)$ for all $n>n_i$ as claimed. We conclude that condition (i) holds for all $R\in\mathcal{S}$ and all $n$ satisfying $n\leq n_{i-1}$ or $n > n_i$. 

Call a tree $R \in\mathcal{S}$ \textit{small}, if the action of $G$ on $R^G$ has orbit growth bounded by $f(n)$ for all $n \geq n_1$, and \textit{large} otherwise. We know that $S_{i-1}$ itself is small. Next we claim that a large tree exists. To see this take the tree $S_{\max}$, consisting of $S_{i-1}$ together with all vertices of level $\leq\ell'$, which lie below $v_i$. Let $\Delta$ be all the descendants of $v_i$ of level $\ell'$.
Let $\mathcal{A}$ be the antichain in $T$ consisting of all vertices of level $\ell$ except the parent of $v_i$, the $p-1$ siblings of $v_i$, and $\Delta$. Then $G_\mathcal{A}$ acts trivially on $\Delta$. Put $Y=S_{\max}^{St(\ell)}$. Then $Y$ is $G_{\mathcal{A}}$-invariant as $G_{\mathcal{A}} \leq St(\ell)$, thus, $Y$ is the union of $G_{\mathcal{A}}$-orbits. If two elements $S_{\max}^g$, $S_{\max}^h$ of $Y$ are in the same $G_{\mathcal{A}}$-orbit, then the colouring of $S_{\max}^g$ and $S_{\max}^h$ coincides on $\Delta$, as $G_{\mathcal{A}}$ acts trivially on $\Delta$. Therefore, the number of different orbits of $G_{\mathcal{A}}$ on $Y$ is at least equal to the number of different colourings of $\Delta$ which occur in one of the trees in $Y$. The number of such colouring equals $(St(\ell):G_{\mathcal{A}})$, as an element $g\in St(\ell)$ stabilizes the colouring on $\Delta$ if and only if it fixes $\Delta$ pointwise, which happens if and only if  $g$ is in $G_{\mathcal{A}}$. We conclude that the number of orbits of $G_{\mathcal{A}}$ is at least

\[
(St(\ell):G_{\mathcal{A}}) = \frac{(G:G_{\mathcal{A}})}{(G:St(\ell))} \geq \frac{(G_{v_i}:G_{\mathcal{A}})}{(G:St(\ell))} \geq \frac{(G_{v_i}|_{S_{v_i}}:G_{\mathcal{A}}|_{S_{v_i}})}{(G:St(\ell))}
\]
where $S_{v_i}$ is the subtree of $T$ with root $v_i$, and $|_{S_{v_i}}$ denotes the restriction to this subtree. By Lemma~\ref{Lem:rigid stabilizer} we have $G_{v_i}|_{S_{v_i}} \cong G$, thus, 
$$(G_{v_i}|_{S_{v_i}}:G_{\mathcal{A}}|_{S_{v_i}}) = (G:St(\ell'-\ell)),$$ 
and therefore, 
\[
o_{(G:G_\mathcal{A})}(G, S_{\max}^G)\geq \frac{(G:G_{\mathcal{A}})}{(G:St(\ell))} = \frac{(G:G_{\mathcal{A}})}{N(\ell)}
\]
as well as
\[
(G:G_{\mathcal{A}}) \geq ((G_{v_i}|_{S_{v_i}}:G_{\mathcal{A}}|_{S_{v_i}}) = N(\ell'-\ell).
\]
Using the choice of $\ell'$ in (\ref{eq: ell definition}) and taking $t=(G:G_{\mathcal{A}})$ we get
\[
o_{(G:G_\mathcal{A})}(G, S_{\max}^G)\geq  \frac{(G:G_{\mathcal{A}})}{N(\ell)} > f\big(\log(G:G_\mathcal{A})\big).
\]
In particular we see that $S_{\max}$ is large. 

Define a graph on $\mathcal{S}$ by drawing an edge between two subtrees if one of them can be obtained from the other by adding or deleting one vertex and its siblings. The graph is connected as it contains $S_{i-1}$. Since there are small and large trees, there exists a small tree $S_i$, which is connected to a large tree $S_i^+$. As $S_i$ is small, it satisfies condition (i).

Since $S_i^+$ is large, there exists some $m_i$ and a subgroup $U_i$ of index $p^{m_i}$ which acts with more than $f(m_i)$ orbits on $(S_i^+)^G$. From our considerations above this can only happen for $n_{i-1} \leq m_i<n_i$. Since $S_i$ and $S_i^+$ are complete, we can apply Lemma~\ref{Lem:continuity of orbit numbers} and find that the number of orbits of $U_i$ on $(S_i^+)^G$ is at most $p$ times the number of orbits of $U_i$ on $S_i^G$, in particular 
\[
o_{p^{m_i}}(G, S_i^G)\geq\frac{1}{p}o_{p^{m_i}}(G, {S_i^+}^G)>\frac{1}{p}f(m_i).
\]
We conclude that $S_i$ satisfies condition (ii) as well.

Now define the tree $S=\bigcup S_i$. We claim that the number of orbits of a finite index subgroup $U$ on $S^G$ equals the number of orbits of $U$ on $S_i^G$, provided that $i$ is sufficiently large depending on $(G:U)$. Since $G$ has the congruence subgroup property, we have that $U$ contains a principal congruence subgroup $St_\ell$. Let $\mathcal{L}$ be the anti chain consisting of all points of level $\ell$. If $S^{g_1}$, $S^{g_2}$ are trees, such that there exists some $u\in U$ with $(S^{g_1u})_\mathcal{L}=(S^{g_2})_\mathcal{L}$, then there is some $c\in St_\ell\leq U$, such that $S^{g_1uc}=S^{g_2}$, that is, $S^{g_1}$ and $S^{g_2}$ are equivalent under $U$. If $i$ is so large that $(S_i)_\mathcal{L} = S_\mathcal{L}$, we obtain that $U$ has as many orbits on $S_i^G$ as on $S^G$. We conclude that the action of $G$ on $S^G$ has orbit growth as described by Theorem~\ref{thm:orbit growth}, and the proof is complete.
\end{proof}

\end{document}